%% file: ms.tex
\documentclass[9pt,shortpaper,twoside,web]{ieeecolor}
\usepackage{generic}
\pdfoutput=1 
\usepackage{cite}
\usepackage{amsmath,amssymb,amsfonts}
\usepackage{algorithmic}
\usepackage{graphicx}
\usepackage{cuted}
\usepackage{flushend}
\usepackage{mathtools}
\usepackage{array}
\usepackage{subcaption}
\mathtoolsset{showonlyrefs}
\usepackage{multirow}
\usepackage{braket}

\usepackage{floatrow}
\newtheorem{innercustomgeneric}{\customgenericname}
\providecommand{\customgenericname}{}
\newcommand{\newcustomtheorem}[2]{%
  \newenvironment{#1}[1]
  {%
   \renewcommand\customgenericname{#2}%
   \renewcommand\theinnercustomgeneric{##1}%
   \innercustomgeneric
  }
  {\endinnercustomgeneric}
}
\newcustomtheorem{customthm}{Theorem}

\usepackage{textcomp}
\def\BibTeX{{\rm B\kern-.05em{\sc i\kern-.025em b}\kern-.08em
    T\kern-.1667em\lower.7ex\hbox{E}\kern-.125emX}}
\newtheorem{theorem}{Theorem}[section]
\newtheorem{definition}{Definition}[section]
\newtheorem{assumption}{Assumption}[section]

\newtheorem{lemma}[theorem]{Lemma}

\usepackage{color,soul}

\begin{document}
\title{Risk Assessment of Stealthy Attacks on Uncertain Control Systems}
\author{Sribalaji C. Anand, Andr\'e M. H. Teixeira, and Anders Ahl\'en
\thanks{This work is supported by the Swedish Research Council under the grant 2018-04396 and by the Swedish Foundation for Strategic Research.}
\thanks{Sribalaji. C. Anand and Anders Ahl\'en are with the Department of Electrical Engineering, Uppsala University, PO Box 65, SE-75103, Uppsala, Sweden. (email: \{sribalaji.anand, anders.ahlen\}@angstrom.uu.se)}
\thanks{Andr\'e M. H. Teixeira is with the Department of Information Technology, Uppsala University, PO Box 337, SE -75105, Uppsala, Sweden. email:andre.teixeira@it.uu.se}}

\maketitle
\begin{abstract}
In this article, we address the problem of risk assessment of stealthy attacks on uncertain control systems. Considering data injection attacks that aim at maximizing impact while remaining undetected, we use the recently proposed output-to-output gain to characterize the risk associated with the impact of attacks under a limited system knowledge attacker. The risk is formulated using a well-established risk metric, namely the maximum expected loss. Under this setups, the risk assessment problem corresponds to an untractable infinite non-convex optimization problem. To address this limitation, we adopt the framework of scenario-based optimization to approximate the infinite non-convex optimization problem by a sampled non-convex optimization problem. Then, based on the framework of dissipative system theory and S-procedure, the sampled non-convex risk assessment problem is formulated as an equivalent convex semi-definite program. Additionally, we derive the necessary and sufficient conditions for the risk to be bounded. Finally, we illustrate the results through numerical simulation of a hydro-turbine power system.
\end{abstract}


\section{Introduction}\label{Intro}

\input{1_Introduction}

\section{Problem background}\label{description}

\input{2A_Description}

\section{Problem formulation}\label{PF}

\input{2B_Problem_Formulation}



\section{Risk assessment for a bounded-rational adversary}\label{rational}

\input{4_Rational}

\section{Numerical example}\label{Example}

\input{5_Example}

\section{Conclusion and Future work}\label{Conclusion}

\input{6_Conclusion}

\bibliographystyle{ieeetr}
\bibliography{ms}
\clearpage
\end{document}

%% file: 1_Introduction.tex
Research in the security of industrial control systems has received considerable attention \cite{dibaji2019systems} due to increased number of cyber-attacks such as the one on Ukrainian power grid \cite{case2016analysis}, Kemuri water company \cite{hemsley2018history}, etc. One of the common recommendations for improving the security of control systems is to follow the risk management cycle: Risk assessment, risk response, and risk monitoring \cite{risk}. This article focuses on risk assessment, the formal definition of which will be introduced later as a function of the attack impact.

Risk is often a combination of attack impact and/or likelihood. For instance, the risk is characterized in terms of average impact in \cite{cardenas2011attacks} for different types of attacks. The consequences of data injection attacks are quantified using the conditional value-at-risk in \cite{CVAR1}. The calculated risk can later be used to compute optimal defense-allocation strategies \cite{teixeira2015secure} and/or design robust controllers/detectors. Risk assessment of combined data integrity and availability attacks against the power system state estimation is conducted in \cite{pan2018cyber}. From this brief discussion, it can be realized that characterizing risk in terms of attack impact and likelihood is critical for the efficient allocation of protection resources. In the literature, the problem of risk assessment of stealthy attacks on uncertain control control has not been addressed. To the best of the authors knowledge, the works that are closely related to this problem are \cite{park2019stealthy}, \cite{harshbarger2020little} and \cite{murguia2020security}.

Firstly, \cite{park2019stealthy} designs a stealthy attack against an uncertain system using disclosure resources. Secondly, \cite{harshbarger2020little} focuses on attack detection based on plant and model mismatch for the adversary. The results of both the above works cannot facilitate the optimal allocation of protection resources. 

Thirdly, \cite{murguia2020security} proposes an impact metric by computing a bound on the reachable set of states by an attacker for a perturbed system dynamics. It also proposes a second metric by computing the distance between the reachable set of states for the adversary and the set of critical states. However it considers a deterministic system.

The advantage of our study is multi-fold. Firstly, we consider a generic modeling framework similar to \cite{Andre1}. Secondly, unlike many previous works, we consider a system description with parametric uncertainty. Thirdly, similar to \cite{milovsevic2020actuator}, we consider an adversarial setup where the system knowledge of the adversary is limited. Finally, we consider a recently proposed impact metric: Output-to-Output Gain ($OOG$) \cite{teixeira2015strategic}. The main advantage of using this impact metric, as opposed to the classical $H_{\infty}$ and $H_{\_}$ metrics is, the $OOG$ metric based design problem focuses on improving detectability only when the impact of the attack is sufficiently high at the same frequency \cite{IFAC}. In other words, the $OOG$ metric is more amenable to risk-optimal system design for increased security. 
Under the described setup, we present the following contributions.
\begin{enumerate}
\item Using $OOG$ as impact metric, and the maximum expected loss as a risk metric, we formulate the risk assessment problem. We observe that the risk assessment problem corresponds to an untractable infinite non-convex robust optimization problem which is NP-hard in general. 
\item We propose a convex SDP which solves the risk assessment problem approximately by sampling the uncertainty set. Additionally, we provide  probabilistic guarantees on the original robust optimization problem. 
\item We derive the necessary and sufficient conditions for the risk metric to be bounded.
\end{enumerate}

The remainder of the article is organized as follows. The uncertain system and the adversary is described in Section \ref{description}. The problem is formulated in Section \ref{PF}, to which an approximate solution is presented in Section \ref{rational}. The efficacy of the proposed optimization problem is illustrated through a  numerical example in Section \ref{Example}. We conclude the article in Section \ref{Conclusion}.

\textit{Notation}: Throughout this article, $\mathbb{R}, \mathbb{C}, \mathbb{Z}$ and $\mathbb{Z}^{+}$ represent the set of real numbers, complex numbers, integers and non-negative integers respectively. A positive (semi-)definite matrix $A$ is denoted by $A \succ 0 (A \succeq 0)$. Let $x: \mathbb{Z} \to \mathbb{R}^n$ be a discrete-time signal with $x[k]$ as the value of the signal $x$ at the time step $k$. Let the time horizon be $[0,N]=\{ k \in \mathbb{Z}^+|\; 0 \leq k\leq N \}$. The $\ell_2$-norm of $x$ over the horizon $[0,N]$ is represented as $|| x ||_{\ell_2, [0,N]}^2 \triangleq \sum_{k=0}^{N} x[k]^Tx[k]$. Let the space of square summable signals be defined as $\ell_2 \triangleq \{ x: \mathbb{Z}^+ \to \mathbb{R}^n |\; ||x||^2_{\ell_2, [0,\infty]} < \infty\}$ and the extended signal space be defined as $\ell_{2e} \triangleq \{ x: \mathbb{Z}^+ \to \mathbb{R}^n | \;||x||^2_{[0,N]} < \infty, \forall N \in \mathbb{Z}^+ \}$. For the sake of simplicity, we represent $||x||^2_{\ell_2,[0,\infty]}$ as $||x||^2_{\ell_2}$. For $x \in \mathbb{R}, \left \lceil {x} \right \rceil$ represents $x$ rounded to the nearest integer greater than or equal to $x$. Let $(\Omega, \mathcal{D}_a, \mathbf{P})$ represent a probability space with sample space $\Omega \subset \mathbb{R}^v$, $\sigma-$algebra $\mathcal{D}_a$, and probability measure $\mathbf{P}$. Let $(\Omega^w, \mathcal{D}_a^w, \mathbf{P}^w)$ represent the $w$-times Cartesian product of $\Omega$ with the $\sigma-$algebra $\mathcal{D}_a^w$ and the probability measure $\mathbf{P}^w = \mathbf{P} \times \dots \times \mathbf{P}$. A point drawn from $(\Omega^w, \mathcal{D}_a^w, \mathbf{P}^w)$ is thus $(\delta_1,\delta_2,\dots\delta_w)$, i.e., $w$ independent elements in $\mathbb{R}^v$ drawn independently from $\Omega$ according to the same probability $\mathbf{P}$.\footnote{In this article, the Cartesian product is considered over the same probability space ($\Delta$). But this can be generalized to arbitrary probability spaces.}. The relative importance of various uncertainty outcomes of an arbitrary function $f(\delta),\; \delta \triangleq \begin{bmatrix} \delta_1 & \dots & \delta_m \end{bmatrix}^T \in \Omega,$  is represented by $\mathbb{P}_{\Omega}(f(\delta))$. An empty set is denoted by $\emptyset$. For any $a, b \in \mathbb{R}$ and $a \geq b$, ${a \choose b} = \frac{a!}{b!(a-b)!}$. The non-zero elements of a vector $s \in \mathbb{R}^{n}$ is denoted by $\text{supp}(s)$. The cardinality of the non-zero elements in a vector $s \in \mathbb{R}^{n}$ is denoted by  $||s||_0 \triangleq |\text{supp}(s)|$, where $|\text{supp}(s)|$ represents the cardinality of the set $\text{supp}(s)$.

%% file: 2A_Description.tex
\begin{figure}
    \centering
    \includegraphics[width=8.4cm]{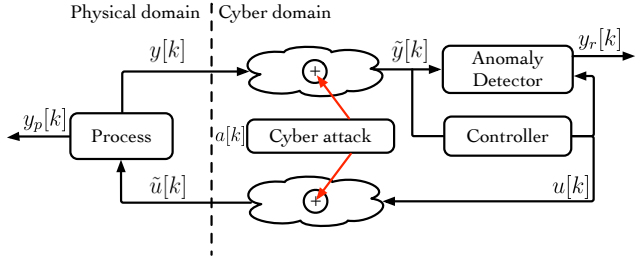}
    \caption{Control system under data injection attack}
    \label{System}
\end{figure}
In this section, we describe the control system structure and the goal of the adversary. Consider the general description of a closed-loop DT LTI system with a process ($\mathcal{P}$), feedback controller ($\mathcal{C}$) and an anomaly detector ($\mathcal{D}$) as shown in Fig. \ref{System} and represented by
\begin{align}
    \mathcal{P}: & \left\{
                \begin{array}{ll}
                    x_p[k+1] &= A^{\Delta}x_p[k] + B^{\Delta} \tilde{u}[k]\\
                    y[k] &= C^{\Delta}x_p[k]\\
                    y_p[k] &= C_Jx_p[k] +D_J\tilde{u}[k]
                \end{array}
                \right. \label{P}\\
    \mathcal{C}: & \left\{
                 \begin{array}{ll}
                     z[k+1] &= A_cz[k]+ B_c\tilde{y}[k]\\
                     u[k] &= C_cz[k] + D_c\tilde{y}[k]
                \end{array}
                \right. \label{C} \\
    \mathcal{D}: & \left\{
                \begin{array}{ll}
                s[k+1] &= A_es[k] +B_e u[k] + K_e \tilde{y}[k]\\
                y_r[k] &= C_es[k] +D_eu[k] + E_e \tilde{y}[k]
                \end{array}
                \right. \label{D}
\end{align}
where  $A^{\Delta} \triangleq A + \Delta A(\delta)$ with $A$ representing the nominal system matrix and $\delta \in \Omega$. Additionally we assume $\Omega$ to be closed,  bounded and to include the zero uncertainty yielding $\Delta A(0) = 0$. The other matrices are similarly expressed. The state of the process is represented by $x_p[k] \in \mathbb{R}^{n_x}$, $z[k] \in \mathbb{R}^{n_z}$ is the state of the controller, $s[k] \in \mathbb{R}^{n_s}$ is the state of the observer, $\tilde{u}[k] \in \mathbb{R}^{n_u}$ is the control signal received by the process, $ u[k] \in \mathbb{R}^{n_u}$ is the control signal generated by the controller, $y[k] \in \mathbb{R}^{n_m}$ is the measurement output produced by the process, $\tilde{y}[k] \in \mathbb{R}^{n_m}$ is the measurement signal received by the controller and the detector, $y_p[k] \in \mathbb{R}^{n_p}$ is the virtual performance output, and $y_r[k] \in \mathbb{R}^{n_r}$ is the residue generated by the detector. In general, the system is considered to have a good performance when the energy of the performance output $||y_p||_{\ell_2}^2$ is small and an anomaly is considered to be detected when the detector output energy $||y_r||_{\ell_2}^2$ is greater than a predefined threshold, say $\epsilon_r$. We explain the reason to adopt such a logic in Appendix \ref{app_explain}. Without loss of generality (w.l.o.g.), we assume $\epsilon_r = 1$ in the sequel.
\subsection{Data injection attack scenario}\label{Attack scenario:sec}
In the system described in \eqref{P}-\eqref{D}, we consider an adversary injecting false data into the sensors or actuators of the plant but not both. Next, we discuss the resources the adversary has access to.
\subsubsection{Disruption and disclosure resources}\label{disclosure:sec} 
The adversary can access (eavesdrop) the control or sensor channels and can inject data. This is represented by
\[ \begin{bmatrix}
\tilde{u}[k]\\
\tilde{y}[k]
\end{bmatrix} 
=
\begin{bmatrix}
{u}[k]\\
{y}[k]
\end{bmatrix} 
+
\begin{bmatrix}
B_a\\
D_a
\end{bmatrix} 
a[k], \left[\begin{array}{c|c}
B_a^T & D_a^T
\end{array}\right] \triangleq \left[\begin{array}{c|c}
E_a^T & 0\\
0^T & F_a^T
\end{array}\right]\]  
where $a[k] \in \mathbb{R}^{n_a}$ is the attack signal injected by the adversary. The matrix $E_a (F_a)$ is a diagonal matrix with $E_a(i,i)=1 \;(F_a(i,i)=1)$, if the actuator (sensor) channel $i$ is under attack and $0$ otherwise.

\subsubsection{System knowledge} 
In general, the system operator knows about the parameters of the controller and detector as s/he is the one who design it. However, it is not known to the adversary. To this end, we consider a realistic adversary whose system knowledge is bounded. We defined this adversary as bounded-rational. Henceforth, we refer to the bounded-rational adversary as a rational adversary.
\begin{definition}[Rational adversary]\label{RA}
An adversary is defined to be rational if it knows all the matrices of \eqref{P}-\eqref{D} but with bounded uncertainty. $\hfill\triangleleft$
\end{definition}

That is, in contrary to the operator, a rational adversary has uncertainty also in the system matrices of the detector and the controller ($A,B,C,C_J,D_J,A_c,B_c,C_c,D_c,A_e,B_e,K_e,C_e,D_e,$ and $E_e$). Defining ${x}[k] \triangleq [ x_p[k]^T \; z[k]^T \; s[k]^T]^T$, the closed-loop system under attack with the performance output and detection output as system outputs becomes
\begin{equation}\label{system:CL:uncertain}
                \begin{array}{ll}
                            {x}[k+1] &= {A}_{cl}^{\Delta}{x}[k] + {B}_{cl}^{\Delta}a[k]\\
                            y_p[k] &= {C}_p^{\Delta}{x}[k] + {D}_p^{\Delta} a[k]\\
                            y_r[k] &= {C}_r^{\Delta}{x}[k] + {D}_r^{\Delta} a[k],\\
                \end{array}
\end{equation}
where the nominal matrices are given by $\left[ \begin{array}{c|c}
A_{cl} & B_{cl}
\end{array}\right] \triangleq $
\[\left[
\begin{array}{c|c}
\begin{matrix}
    A+BD_cC & BC_c & 0\\
    B_cC & A_c & 0\\
    (B_eD_c +K_e)C & B_eC_c & A_e
    \end{matrix} & \begin{matrix}
    BB_a + BD_cD_a \\ B_cD_a \\ (B_eD_c+K_e)D_a 
    \end{matrix}
\end{array}\right]
\]
\begin{align*}
    {C}_p &\triangleq \begin{bmatrix}
    C_J+D_JD_cC & D_JC_c & 0
    \end{bmatrix},\quad {D}_p \triangleq D_J(D_cD_a+B_a),\\
    {C}_r &\triangleq \begin{bmatrix}
    (D_eD_c + E_e)C & D_eC_c & C_e
    \end{bmatrix}, {D}_r \triangleq (D_eD_c +E_e)D_a.
\end{align*}

Although the uncertainty is present in the adversarial system knowledge, the bounds of the uncertainty is chosen by the defender based on the his/her belief of the adversaries resources.  Next, we assume the following for clarity. 
\begin{assumption}\label{assume_stable}
The closed-loop system \eqref{system:CL:uncertain} is stable $ \forall \;\delta \in \Omega$.$\hfill\triangleleft$
\end{assumption}
\begin{assumption}\label{Ba}
The input matrix has full column rank i.e., $ \nexists \;s \in \mathbb{R}^{n_a} \neq 0$ such that $B_{cl}^\Delta s =0$.$\hfill\triangleleft$
\end{assumption}
\subsubsection{Attack goals and constraints.}
Given the resources the adversary has access to, the adversary aims at disrupting the system's behavior while staying stealthy. The system disruption is evaluated by the increase in energy of the performance output whereas, the adversary is stealthy if the energy of the detection output is below a predefined threshold $\epsilon_r=1$. 
\begin{definition}[Stealthy attack]
An attack vector $a$ is said to be stealthy if it satisfies $||y_r||_{\ell_2}^2 \leq 1$.$\hfill\triangleleft$
\end{definition}
We discuss the attack policy for a deterministic system next.

\subsection{Optimal attack policy for the nominal system}

From the previous discussions, it can be understood that the goal of the adversary is to maximize the performance cost while staying undetected. When the system \eqref{system:CL:uncertain} is deterministic, i.e., $\Omega=\emptyset$, \cite{teixeira2015strategic} formulates the attack policy of the adversary as the following non-convex optimization problem
\begin{equation}\label{opti:o2o:uncertain}
\begin{aligned}
||\Sigma||_{\ell_{2e},y_p \leftarrow y_r}^2 \triangleq \sup_{a \in \ell_{2e}} \; & ||y_p||^2_{\ell_{2}} \\
\textrm{s.t.} \quad & ||y_r||_{\ell_2}^2 \leq 1,\; x[0]=0,\;x[\infty]=0,
\end{aligned}
\end{equation}
where $||\Sigma||_{\ell_{2e},y_p \leftarrow y_r}^2$ represents the $OOG$ that characterizes the disruption caused by the attack signal $a$. In \eqref{opti:o2o:uncertain}, the constraint $x[0]=0$ is introduced because the system is assumed to be at equilibrium before the attack. 

\begin{assumption}\label{equil_begin}
The closed-loop system \eqref{system:CL:uncertain} is at equilibrium $x[0]=0$ before the attack commences.$\hfill\triangleleft$
\end{assumption}

We assume that the adversary has finite energy (similar to $H_{\infty}/H_{\_}$ approaches) and thus attacks the system for a sufficiently long but finite amount of time, say $T$. Although this time $T$ is unknown \textit{a priori}, it can be ensured that $ a[\infty] \triangleq \lim_{k\to \infty} a[k] = 0$ by setting $x[\infty] \triangleq \lim_{k\to \infty} x[k] = 0$.
\begin{lemma}\label{vanish_lemma}
If $x[\infty]=0$, then it holds that $a[\infty]=0$.
\end{lemma}
\begin{proof}
If $x[\infty]=0$, it follows that 
\begin{equation}
    \lim_{k\to \infty} \Vert x[k+1] - A_{cl}x[k]\Vert = 0 =     \lim_{k\to \infty} \Vert B_{cl}a[k] \Vert = \Vert B_{cl}a[\infty]\Vert.
\end{equation}
From \textit{Assumption \ref{Ba}}, we known that $B_{cl}a[\infty] = 0 $ holds only when $a[\infty]=0$. This concludes the proof.
\end{proof}

For the above discussed reason, the constraint $x[\infty]=0$ is introduced in the optimization problem \eqref{opti:o2o:uncertain}. In this article, we focus on attacks which satisfy the constraint $x[\infty]=0$, and define them as state-vanishing attacks.
\begin{definition}[State-vanishing attack]\label{vanish_attack_def}
An attack vector $a$ is defined to be state-vanishing if applying $a$ to \eqref{system:CL:uncertain} generates the state vector $x$ which satisfies $\lim_{k \to \infty}x[k] = 0$.$\hfill\triangleleft$
\end{definition}

In the literature, such characterization of impact of stealthy attacks \eqref{opti:o2o:uncertain} has only been studied for deterministic systems. The remainder of this article is focused on discussing methods to quantify the impact in terms of risk on the uncertain system \eqref{system:CL:uncertain}.

The concept of risk is conventionally adopted to address decision-making in an uncertain environment \cite[Chapter 14]{ferrari2021safety}. Since we also focus on decision-making (impact assessment) under uncertainty, it is useful to adopt tools from the risk management framework. Thus, before introducing the problem formulation, a brief introduction to risk management and risk metrics is provided as it helps in the problem formulation.

%% file: 2B_Problem_Formulation.tex
The framework of risk management can help the system operator answer the following (but not limited) questions: $(i)$ Which components of the system are critical to the operation of the system? $(ii)$ What disruption can be expected from attacks and $(iii)$ Which resources should be protected and how? Thus, to use the risk management framework for the benefit of the system operator (to estimate system disruption $(ii)$), we will focus on risk and its consequences. To quantify the risks of data injection attacks on an uncertain system, we start by defining an impact random variable a function of the system uncertainty and the attack vector.

\begin{definition}[Impact random variable]\label{RV}
Let the random variable $X^A(\cdot)$ be defined as
\begin{equation}\label{newimpact}
\begin{aligned}
X^A(a,\delta) \triangleq & \;\Vert y_p(\delta)\Vert_{\ell_2}^2 \times  \mathbb{I}\bigg( \Vert y_r(\delta)\Vert_{\ell_2}^2 \leq 1,x({\delta})[\infty]=0 \bigg)
\end{aligned}
\end{equation}
where $X^A(\cdot)$ is the impact caused on the system \eqref{system:CL:uncertain} with uncertainty $\delta \in \Omega$ by the attack vector $a \in  \ell_{2e}$, $\mathbb{I}$ is the indicator function,  $y_p(\delta), y_r(\delta)$ and $x({\delta})$ are the performance, residue output and state of the system with the isolated uncertainty $\delta$. Here, the signals $y_p(\delta), y_r(\delta)$ and $x(\delta)$ are also functions of the attack vector $a$.$\hfill\triangleleft$
\end{definition}

With the random variable defined in \textit{Definition \ref{RV}}, we next formulate the risk assessment problem in for a rational adversary.  Consider the data injection attack scenario where only the bounds of the parametric uncertainty set $\Omega$ is known to the adversary. Although, the adversary can determine the attack vector which maximizes the expected loss over the entire uncertainty set $\Omega$ i.e., the adversary maximizes the function
\begin{equation}
\begin{aligned}
 ||\bar{\Sigma}(\delta)||_{\ell_{2e},y_p \leftarrow y_r}^2 \triangleq &\; \mathbb{E}_{\Omega}\left\{X^A(a,\delta)\right\}.
\end{aligned}
\end{equation}

This setup is common in game-theoretic approaches \cite{lye2005game} where the players do not know the strategy of the other players and thus play by maximizing its expected return over all the strategies of the other players. Similarly, since the adversary has limited system knowledge, s/he chooses an attack policy which maximizes the expected loss of the system operator over the set $\Omega$ whilst remaining stealthy. This strategy of maximum expected loss can be defined as follows.
\begin{definition}[Maximum Expected Loss]\label{expect_loss_def}
The maximum expected loss ($\mathbb{EL}$) associated with the impact-random variable $X^A(\cdot)$, defined in \eqref{newimpact}, is defined as
\begin{align}
\mathbb{EL}[X^A] \triangleq \sup_{a \in \ell_{2e}} \; \mathbb{E}_{\Omega} \left\{X^A(\delta,a) \right\},
\end{align}
where $X^A(\delta,a)$ is the loss on scenario $\delta$ and $\mathbb{E}_{\Omega}$ represents the expectation operator over the set $\Omega$.$\hfill\triangleleft$
\end{definition}

Thus by determining the attack vector that solves for maximal expected loss, one can ensure that the attack vector is stealthy with respect to all uncertainties whilst maximizing the performance loss. Using \textit{Definition \ref{expect_loss_def}}, the risk associated with the impact caused by a bounded-\textbf{R}ational \textbf{A}dversary can be characterized as

\begin{equation}\label{w1}
\begin{aligned}
\gamma_{RA}  &\triangleq \sup_{a \in \ell_{2e}} \mathbb{E}_{\Omega} (X^A(a,\delta)).
\end{aligned}
\end{equation}


Since the operator $\mathbb{E}$ in \eqref{w1} operate over the continuous space $\Omega$, \eqref{w1} is computationally intensive or in general NP-hard \cite[Section 3]{ben1998robust}. Besides, the problem is also non-convex. In the remainder of this article, we discuss methods to solve the optimization problem approximately and efficiently.

%% file: 4_Rational.tex
In this section, we focus on describing a scenario-based approach to the optimization problem \eqref{w1}.
\subsection{Approximating the uncertainty set }
To recall, we are interested in determining the maximum expected loss associated with the impact caused by a rational adversary. Unfortunately, this problem is computationally intensive or in general NP-hard. Thus, as a first step toward solving \eqref{w1}, we approximate the objective function in \textit{Lemma \ref{Expectation_sample_lemma}}.
\begin{lemma}\label{Expectation_sample_lemma}
Let $\delta_i$ be sampled uncertainties from $\Omega$. Let us define
\begin{equation}
    \hat{\mathbb{E}}_{\Omega_{N_2}}(X^A(a,\delta)) \triangleq \frac{1}{N_2}\sum_{i=1}^{N_2} X^A(a_i,\delta_i).
\end{equation}
Then, it holds that $\lim_{N_2 \to \infty} \hat{\mathbb{E}}_{\Omega_{N_2}}(X^A(a_i,\delta_i)) = \mathbb{E}_{\Omega}(X^A(a,\delta)).$
\end{lemma}
\begin{proof}
The proof follows from applying \cite[{Theorem 7.2}]{tempo2012randomized} to approximate the expectation operator in \eqref{w1}.
\end{proof}

\textit{Lemma \ref{Expectation_sample_lemma}} states that the continuous set $\Omega$ can be approximated by a discrete set $\Omega_{N_2}$ of cardinality $N_2$. The approximation becomes more accurate as $N_2 \to \infty$. This approximation introduces a curse of dimensionality since, to obtain a good estimate of the risk and to obtain a well feasible attack vector, the required number of samples is too large for practical implementation. To circumvent this practical issue, we next show that an attack vector obtained by solving \eqref{w1} with a discrete uncertainty set as mentioned in \textit{Lemma \ref{Expectation_sample_lemma}} is partially feasible to the original optimization problem \eqref{w1} with a continuous uncertainty set. 

It might not be immediately apparent that the notion of feasibility applies to \eqref{w1} since there are no external constraints present. Thus, we begin by simplifying the optimization problem \eqref{w1}.
\begin{lemma}\label{pp31}
The optimization problem \eqref{w1} is equivalent to the optimization problem 
\begin{equation}\label{pp3}
\begin{aligned}
 \sup_{a \in \ell_{2e},\beta \in [0,\;1]} \; & {\mathbb{E}_{\Omega}} \left\{||y_p(\delta)||_{\ell_2}^2\; \vert \; \begin{pmatrix}
||y_r(\delta)||_{\ell_2}^2 \leq1\\
x(\delta)[\infty]=0
\end{pmatrix} \right\} (1-\beta) \\
\textrm{s.t.} \quad & {\mathbb{P}_{\Omega}}
\begin{pmatrix}
||y_r(\delta)||_{\ell_2}^2 \leq1\\
x(\delta)[\infty]=0
\end{pmatrix}\geq 1-\beta.
\end{aligned}
\end{equation}
\end{lemma}
\begin{proof}
Consider the function $X^A(a,\delta)$ in \eqref{newimpact}. By expanding its indicator function, we can write $\mathbb{E}_{\Omega}(X^A(a,\delta))$ as
\begin{align}\label{w3}
\nonumber & {\mathbb{E}_{\Omega}} \left\{ ||y_p(\delta)||_{\ell_2}^2\; \vert \begin{pmatrix}
||y_r(\delta)||_{\ell_2}^2 \leq1\\
x(\delta)[\infty]=0
\end{pmatrix} \right \}{\mathbb{P}_{\Omega}}\begin{pmatrix}
||y_r(\delta)||_{\ell_2}^2 \leq1\\
x(\delta)[\infty]=0
\end{pmatrix}
\end{align}
Using the above definition in \eqref{w1}, we obtain $\gamma_{RA}$ as
\begin{equation}
\sup_{a \in \ell_{2e}} {\mathbb{E}_{\Omega}} \left\{ ||y_p(\delta)||_{\ell_2}^2\; \vert \;\begin{pmatrix}
||y_r(\delta)||_{\ell_2}^2 \leq1\\
x(\delta)[\infty]=0
\end{pmatrix} \right \}{\mathbb{P}_{\Omega}}\begin{pmatrix}
||y_r(\delta)||_{\ell_2}^2 \leq1\\
x(\delta)[\infty]=0
\end{pmatrix}
\end{equation}
which can be rewritten as \eqref{pp3}. This concludes the proof.
\end{proof}

\textit{Lemma \ref{pp31}} uncovers the constraints present in the optimization problem \eqref{w1}. We can now discuss the notion of feasibility in regards to the optimization problem \eqref{pp3}. So, we continue by simplifying \eqref{pp3} as follows. In reality, $\beta$ represents the fraction of the uncertainty set with respect to which the adversary is not stealthy. Let us assume that the adversary is injecting a maximally robust attack, i.e., the attack vector is stealthy w.r.t. all uncertainties. Thus we could set $\beta =0$. Motivated by the above argument, we rewrite \eqref{pp3} as
\begin{equation}\label{fi1}
\begin{aligned}
\sup_{a \in \ell_{2e}} \; & {\mathbb{E}_{\Omega}} \left\{||y_p(\delta)||_{\ell_2}^2\; \vert \;\begin{pmatrix}
||y_r(\delta)||_{\ell_2}^2 \leq1\\
x(\delta)[\infty]=0
\end{pmatrix}\right\} \\
\textrm{s.t.} \quad & 
\begin{pmatrix}
||y_r(\delta)||_{\ell_2}^2 \leq1\\
x(\delta)[\infty]=0
\end{pmatrix} \forall \delta \in \Omega
\end{aligned}
\end{equation}
Recalling the approximation result in \textit{Lemma \ref{Expectation_sample_lemma}}, assume that $\Omega$ is replaced with a discrete uncertainty set $\Omega_{N}$, so that \eqref{fi1} is approximated by
\begin{equation}\label{uncertain:adversary}
\begin{aligned}
\gamma_{RA}(N_2) = \sup_{a \in \ell_{2e}} & \hat{\mathbb{E}}_{\Omega_{N_2}} \left\{||y_p(\delta)||_{\ell_2}^2\; \vert \;\begin{pmatrix}
||y_r(\delta)||_{\ell_2}^2 \leq1\\
x(\delta)[\infty]=0
\end{pmatrix}\right\} \\
\textrm{s.t.} & 
\begin{pmatrix}
||y_r(\delta)||_{\ell_2}^2 \leq1\\
x(\delta)[\infty]=0
\end{pmatrix} \forall \delta \in \Omega_{N_2}.
\end{aligned}
\end{equation}

Let the resulting optimal attack vector be denoted by $a_{N_2}^*$. Then the following theorem provides \textit{a posteriori} results on the feasibility of the attack vector $a_{N_2}^*$ to \eqref{fi1}.

\begin{theorem}\label{Thm_sampling}
Let the number of samples $N_2$ and the confidence level $\lambda \in (0,1)$ be predefined constants. Define $\epsilon(\cdot)$ such that 
\begin{equation}\label{sample:constraint2}
\epsilon(N_2) = 1 \;,\; \sum_{k=0}^{N_2-1} {N_2\choose k}(1 - \epsilon(k))^{N_2-k} = \lambda.
\end{equation}
Let $s_{N_2}^*$ represent the cardinality of the support subsample for $(\delta_1,\dots,\delta_{N_2})$ (see \cite[\textit{Definition 2}]{campi2018general}). Then it holds that 
\begin{equation}
\mathbb{P}^{N_2}\{ \mathbb{P}_{\Omega}\{\delta \in \Omega \;\vert\; a_{N_2}^* \notin \Theta\}> \epsilon(s_{N_2}^*)\}\leq \lambda,
\end{equation}
where $a_{N_2}^*$ is the argument of the optimization problem \eqref{uncertain:adversary} and $\Theta$ is defined as
\begin{align}\label{theta}
    \Theta &\triangleq \bigcap_{\delta \in \Omega} \Theta_{\delta}, \text{where} \;\Theta_{\delta} \triangleq \left\{ a \in \ell_{2e} \vert 
\begin{pmatrix}
||y_r(\delta)||_{\ell_2}^2 \leq1\\
x(\delta)[\infty]=0
\end{pmatrix}
\right\}
\end{align}
\end{theorem}
\begin{proof}
See Appendix \ref{Main_thm}.
\end{proof}

In words, \textit{Theorem \ref{Thm_sampling}} states that the attack vector obtained by solving \eqref{uncertain:adversary} is stealthy and state-vanishing for all the closed-loop system of the form \eqref{system:CL:uncertain} with uncertainties belonging to the set $\Omega$ except for the fraction $\epsilon(s_{N_2}^*)$ of $\Omega$. It also states that the accuracy ($\epsilon(\cdot)$) and confidence ($\lambda$) of the solution are independent of the distribution. This result is a direct consequence of \cite[\textit{Theorem 1}]{campi2018general}.

In conclusion, it follows that \eqref{fi1} can be solved approximately with a discrete set $\Omega_{N_2}$ of arbitrary but bounded cardinality. Thus, the next section will focus on solving the optimization problem \eqref{uncertain:adversary}.
\subsection{Risk assessment}
The optimization problem \eqref{uncertain:adversary} has two main disadvantages namely $(i)$ it is a non-convex optimization problem, and $(ii)$ its constraints lie in the infinite-dimensional attack space. To resolve these disadvantages, we adopt S-procedure and dissipative system theory and revisit the optimization problem \eqref{uncertain:adversary} in the theorem below. To begin with, given a sampled uncertainty $\delta_i \in \Omega$, we define $\tilde{\Sigma}_{p,i} \triangleq ({A}_{cl,i}, {B}_{cl,i}, {C}_{p,i}, {D}_{p,i})$ and  $\tilde{\Sigma}_{r,i} \triangleq ({A}_{cl,i}, {B}_{cl,i}, {C}_{r,i}, {D}_{r,i})$ with $y_p(\delta_i)=y_{pi},y_r(\delta_i)=y_{ri}$ and $x(\delta_i)=x_{i}$ as the outputs and states of $\tilde{\Sigma}_{p,i}$ and  $\tilde{\Sigma}_{r,i}$ correspondingly. Now we are ready to present the main theorem of this section.

\begin{theorem}\label{Thm2}
Let $N_2$ be a predefined constant. The maximum expected loss \eqref{uncertain:adversary} associated with the impact caused by a rational adversary injecting a maximally robust attack vector on \eqref{system:CL:uncertain} can be obtained by solving the convex SDP \footnote{With an abuse of notation, we denote that every element of the vector $\gamma$ is non-negative by $\gamma \geq 0$}
\begin{equation}\label{problem:coupled:final}
\begin{aligned}
\min_{\gamma \geq 0, P = P^T}  & \mathbf{1}^T\begin{bmatrix}
\gamma_1 & \dots & \gamma_{N_2}
\end{bmatrix}^T \\
\textrm{s.t.} \quad &
\begin{bmatrix}
\Bar{A}^TP\Bar{A}-P &\Bar{A}^TP\Bar{B}\\ \Bar{B}^TP\Bar{A}&\Bar{B}^TP\Bar{B}\end{bmatrix}+ \Psi(\gamma) \preceq 0,
\end{aligned}
\end{equation}
where $\Psi(\gamma) \triangleq \frac{1}{N_2}
\begin{bmatrix}
\bar{C}_{p}^T\\
\bar{D}_{p}^T
\end{bmatrix}
\begin{bmatrix}
\bar{C}_{p} & \bar{D}_{p}
\end{bmatrix}
- \begin{bmatrix}
\bar{C}_{r}^T\\
\bar{D}_{r}^T
\end{bmatrix}\Gamma(\gamma)
\begin{bmatrix}
\bar{C}_{r} & \bar{D}_{r}
\end{bmatrix}$,
\begin{align}\label{balaji}
\left[
\begin{array}{c|c}
\bar{A} & \bar{B}\\
\hline
\\[\dimexpr-\normalbaselineskip+2pt] \bar{C}_p & \bar{D}_p\\
\hline
\\[\dimexpr-\normalbaselineskip+2pt] \bar{C}_r & \bar{D}_r
\end{array}
\right] = 
\left[ 
\begin{array}{c|c}
\begin{matrix}
A_{cl,1} & \dots & 0 \\
\vdots & \ddots & \vdots\\
0 & \dots & A_{cl,N_2}
\end{matrix}     & \begin{matrix}
B_{cl,1} \\ \vdots \\ B_{cl,N_2}
\end{matrix} \\
\hline
\\[\dimexpr-\normalbaselineskip+2pt] \begin{matrix}
C_{p,1} & \dots & 0\\
\vdots  & \ddots & \vdots \\
0 & \dots & C_{p,N_2}
\end{matrix} & \begin{matrix}
D_{p,1}^T \\ \vdots \\ D_{p,N_2}^T
\end{matrix}\\
\hline
\\[\dimexpr-\normalbaselineskip+2pt] \begin{matrix}
C_{r,1} & \dots & 0\\
\vdots  & \ddots & \vdots \\
0 & \dots & C_{r,N_2}
\end{matrix} & \begin{matrix}
D_{r,1}^T \\ \vdots \\ D_{r,N_2}^T
\end{matrix}
\end{array}
\right],
\end{align} 
and $\Gamma(\gamma) =I_{n_r} \bigotimes \text{diag}(\gamma_1,\dots.\gamma_{N_2})$. The dimension of each matrix is given in TABLE \ref{matrix_dimention}.
\begin{table}
\centering
\begin{tabular}{|| m{3em} | m{6em}|| m{3em} | m{6em}||}
 \hline
Matrix & Dimension & Matrix & Dimension  \\
\hline
 $\bar{A}$ &  $n_cN_2 \times n_cN_2$ & $\bar{B}$ & $n_cN_2 \times n_a$\\
 $\bar{C}_p$ & $n_pN_2\times n_cN_2$ & $\bar{D}_p$ & $n_pN_2\times n_a$\\
 $\bar{C}_r$ & $n_rN_2\times n_cN_2$ & $\bar{D}_r$ & $n_rN_2\times n_a$\\
 \hline
\end{tabular}
\caption{Dimension of matrices,$\;n_c\triangleq n_x+n_z+n_s$}
\label{matrix_dimention}
\end{table}
\end{theorem}
\begin{proof}
See Appendix \ref{App2}.
\end{proof}

The optimization problem \eqref{uncertain:adversary} is the primal problem with its optimizer being the attack vector $a$. This optimization problem is non-convex. With the help of S-procedure and dissipative system theory, \eqref{uncertain:adversary} is converted to its equivalent dual SDP form \eqref{problem:coupled:final} with its optimizer $\gamma, P$, which is convex. This equivalency also helps us to conclude that the duality gap is zero. The necessary and sufficient conditions for the value of \eqref{problem:coupled:final} to be finite is given \textit{Lemma \ref{Bound_RA}}.

\begin{lemma}[Boundedness]\label{Bound_RA}
Consider $N_2$ i.i.d. realizations of the closed-loop system \eqref{system:CL:uncertain}, each with an uncertainty $\delta_i$. The optimal value of \eqref{problem:coupled:final} with the above mentioned system realizations is bounded if and only if the system with $\bar{\Sigma}_p=(\Bar{A},\Bar{B},\bar{C}_p,\bar{D}_p)$ and $\bar{\Sigma}_r=(\Bar{A},\Bar{B},\bar{C}_r,\bar{D}_r)$ satisfy one of the following:
\begin{enumerate}
    \item The system $\bar{\Sigma}_r$ has no zeros on the unit circle.
    \item The zeros on the unit circle of the system $\bar{\Sigma}_r$ (including multiplicity and input direction) are also zeros of $\bar{\Sigma}_p$.
\end{enumerate}
\end{lemma}
\begin{proof}
See Appendix \ref{bound_2}.
\end{proof}
Let the outputs of $\bar{\Sigma}_p$ and $\bar{\Sigma}_r$ be represented by $\bar{y}_r$ and $\bar{y}_p$ respectively. Then, in words, \textit{Lemma \ref{Bound_RA}} states that the maximum expected loss \eqref{uncertain:adversary} is bounded if either, there does not exist an attack vector which makes the output $\bar{y}_r$ identically zero, or for all attack vectors which yields $\bar{y}_r$ identically $0$, it also yield $\bar{y}_p$ identically zero. 

It is important to study the conditions for unbounded risk because, if the conditions of the lemma do not hold, it means that there exists an attack vector that can cause a very huge impact whose upper-bounded cannot be determined and remain stealthy. However, as the conditions of the lemma are necessary and sufficient, the operator can alter the system matrices so that the conditions hold and consequently reduce the vulnerability of the system to such attacks. In the next section, we illustrate the results with a numerical example.

%% file: 5_Example.tex
Consider a power generating system \cite[Section 4]{park2019stealthy} as shown in Fig. \ref{turbine} and represented by 
\begin{align}
\label{power_AB} \begin{bmatrix}
\dot{\eta}_1\\ \dot{\eta}_2 \\ \dot{\eta}_3
\end{bmatrix} &= 
\begin{bmatrix}
\frac{-1}{T_{lm}} & \frac{K_{lm}}{T_{lm}} & \frac{-2K_{lm}}{T_{lm}}\\
0 & \frac{-2}{T_h} & \frac{6}{T_h}\\
\frac{-1}{T_g R} & 0 & \frac{-1}{T_g}
\end{bmatrix}
\underbrace{\begin{bmatrix}
{\eta}_1\\ {\eta}_2 \\ {\eta}_3
\end{bmatrix}}_{\eta}
+ \begin{bmatrix}
0\\ 0 \\ \frac{1}{T_g}
\end{bmatrix}
\Tilde{u}\\
\label{power_C} y &= \underbrace{ \begin{bmatrix}
1 & 0 & 0 
\end{bmatrix}}_{C}\begin{bmatrix}
\eta_1\\ \eta_2 \\ \eta_3
\end{bmatrix},\;\;
y_p = \underbrace{
\begin{bmatrix}
1 & 0 & 0\\ 0 & 1 & 0
\end{bmatrix}}_{C_p}\begin{bmatrix}
\eta_1\\ \eta_2 \\ \eta_3
\end{bmatrix}.
\end{align}
Here $\eta \triangleq [df; dp + 2 dx; dx]$, $df$ is the frequency deviation in \mbox{Hz}, $dp$ is the change in the generator output per unit (\mbox{p.u.}), and $dx$ is the change in the valve position \mbox{p.u.}. The parameters of the plant are listed in TABLE \ref{param}. The constants $T_{lm}, T_h$, and $T_g$ represent the time constants of load and machine, hydro turbine, and governor, respectively, and $R (Hz/p.u.)$ is the speed regulation due to the governor action. The constant $K_{lm}$ represents the steady state gain of the load and machine. The DT system matrices $(A^{\Delta},B^{\Delta},C^{\Delta},D^{\Delta})$ are obtained by discretizing the plant \eqref{power_AB}-\eqref{power_C} using zero-order hold with a sampling time $T_s$. The plant is stabilized with an output feedback controller of the form \eqref{C} with $D_c=19$. The detector is an observer-based residue generator of the form \eqref{D} with matrices $A_e = (A_d-K_eC_d), B_e = B_d, C_e=C_d$ where $K_e=\begin{bmatrix} 0.17 & -2.83 & -7.43 \end{bmatrix}^T$. In this particular setup, the adversary is considered to attack only the actuator, i.e., $B_a = 1$ and $D_a =0$. The other unspecified matrices are zero. Only the matrix $A^{\Delta}$ is a function of the variable $T_h$ and thus is uncertain. Next, considering \eqref{power_AB}-\eqref{power_C} where $T_h$ is uncertain and is uniformly distributed, we discuss the risk associated with the impact caused by a stealthy adversary. 

\begin{figure}
    \centering
    \includegraphics[width=8.4cm]{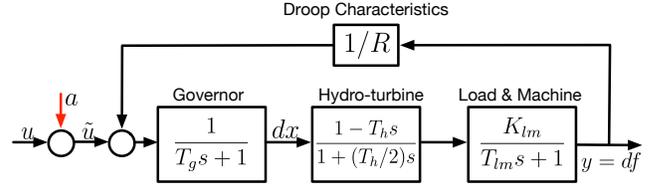}
    \caption{Power generating system with a hydro turbine.}
    \label{turbine}
\end{figure}
\begin{table}
\centering
\begin{tabular}{||c | c || c | c|| c | c ||} 
 \hline
 $K_{lm}$ & 1 & $T_{lm}$ & 6 &  $T_g$ & 0.2 \\
 \hline
 $T_{h}$ & $[4\;\;6]$ & $T_s$ & 0.1 & $R$ & 0.05\\
 \hline
\end{tabular}
\caption{System Parameters}
\label{param}
\end{table}

In view of \textit{Lemma \ref{Expectation_sample_lemma}}, we choose $N_2=21$ to approximate the set $\Omega$. With this approximation, by solving the convex SDP \eqref{problem:coupled:final}, we obtain $\gamma_{RA}(N_2) = 617.267$. To recall, $\gamma_{RA}$ represents the maximum expected performance loss of the system operator. In this implementation, the value of risk was obtained for $\beta=0$ since we considered a maximally robust adversary. 

Next, we discuss the validity of the approximation $\Omega_{N_2}$. For varying values of $N_2$, the number of non-zero $\gamma_{i}$s obtained while solving \eqref{problem:coupled:final} is shown in the first two columns of TABLE \ref{guarentees_table}. In view of \textit{Theorem \ref{Thm_sampling}}, if we solve the problem \eqref{problem:coupled:final} with an arbitrary $N_2$, we can provide guarantees on the optimization problem \eqref{w1} as shown in column $3$ of TABLE \ref{guarentees_table}. Here, the function $\epsilon(\cdot)$ is evaluated according to \cite[(7)]{campi2018general} where $\lambda=10^{-2}$ and $s_{N_2}^* = |\text{supp}(\gamma)|$\footnote{The confidence is denoted by $\lambda$ here whereas it is denoted by $\beta$ in \cite{campi2018general}}. And $\epsilon(\cdot)$ should be interpreted as follows: The attack vector obtained by solving \eqref{problem:coupled:final} with $N_2$ samples, with probability $1-\lambda$, will be at-most feasible for the fraction $(1-\epsilon(\cdot))\Omega$ of the set $\Omega$.

Since $\epsilon(\cdot)$ represents the fraction of the uncertainty set to which the attack vector is infeasible, it is intuitive to expect this value to be close to zero. Numerically we have observed from TABLE \ref{guarentees_table} that $\text{supp}(\gamma)$ is always one. So, assuming again that $s_{N_2}^*=1$, the number of samples required to guarantee that the attack vector, with a probability $1-\lambda$, is feasible for $1-\epsilon(\cdot)$ of the uncertainty set can be obtained by picking $N_2$ such that \eqref{N2_closed_form} holds \cite[(7)]{campi2018general}
\begin{equation}\label{N2_closed_form}
  \sqrt[\leftroot{-2}\uproot{2}N_2-1]{\frac{\lambda}{N_2^2}} = 1-\epsilon(\cdot).  
\end{equation}
Thus \eqref{N2_closed_form} gives an idea of how $N_2$ increases as $\epsilon(\cdot)$ decreases. And consequently gives an idea of the scalability of the proposed approach as the dimension of the matrix inequality \eqref{problem:coupled:final} depends on $N_2$. We also depict the computation complexity by providing the time taken to solve \eqref{problem:coupled:final} in the last column of TABLE \ref{guarentees_table}. 

\begin{table}
\centering
\begin{tabular}{|| c | c | c | c | c||}
 \hline
 $N_2$ & $\vert\text{supp}(\gamma)\vert$ & $\epsilon(s_{N_2}^*)$ & $\gamma_{RA}(N_2)$ & Time (seconds)\\
 \hline 
 8 & 1 & 0.7141 & 638.04 & 17\\
 \hline
15 & 1 & 0.5112 & 628.60 & 116\\
 \hline
 21 & 1 & 0.4142 & 617.26 & 578\\
 \hline
\end{tabular}
\caption{Rational adversary - a posteriori Guarantees}
\label{guarentees_table}
\end{table}


%% file: 6_Conclusion.tex
In this article, we study the problem of risk assessment on uncertain control systems under a bounded-rational adversary. Using the OOG as an impact metric, we formulated the risk assessment problem and observe that it corresponds to non-convex robust optimization problem. A scenario-based approach was used to relax the robust optimization problem to their sampled counterpart. Using dissipative system theory, the non-convex sampled problem in signal space were converted to their convex dual problems in matrix inequalities. Detailed proof on the zero-duality gap was provided using the S-procedure. We additionally provide necessary and sufficient conditions for risks to be bounded which highlights the important role of uncertainty and how it is incorporated in attack scenarios. The results are depicted with numerical examples. Future work includes investigating the risk assessment problem where the uncertainty set can be approximated as a polytopic set.

\renewcommand{\thesection}{A.\arabic{section}}
\setcounter{section}{0}  
\renewcommand{\thedefinition}{A.\arabic{section}.\arabic{definition}}
\setcounter{definition}{0}
\renewcommand{\thetheorem}{A.\arabic{section}.\arabic{theorem}}
\setcounter{theorem}{0}
\renewcommand{\theassumption}{A.\arabic{section}.\arabic{assumption}}
\setcounter{assumption}{0}
\renewcommand{\theremark}{A.\arabic{section}.\arabic{remark}}
\setcounter{remark}{0}

\section*{Appendix}
\section{Explaining the performance criterion}\label{app_explain}
Consider an LTI system which was described in \eqref{P} without attack. Let us consider that, the operator is evaluating the performance of the system by a linear-quadratic (LQ) cost with cost matrix $Q \succeq 0$ on the states and  cost matrix $R \succ 0$ on the control input. Here, with a slight loss of generality, we consider that the cost on the cross-term is zero. Next, let us define
    \begin{equation}
        y_p[k] = \underbrace{\begin{bmatrix}
        \sqrt{Q} \\ 0
        \end{bmatrix}}_{C_J} x[k] + \underbrace{\begin{bmatrix}
        0 \\ \sqrt{R}
        \end{bmatrix}}_{D_J} u[k] = \begin{bmatrix}
        \sqrt{Q} & 0\\
        0 & \sqrt{R}
        \end{bmatrix} \begin{bmatrix}
        x[k] \\ u[k] 
        \end{bmatrix}
    \end{equation}
    Here, since the matrix $Q$ and $R$ are positive definite and positive semi-definite, their square root exists. Then, evaluating this LQ cost is equivalent to evaluating the $\ell_2$ norm of $y_p[k]$. That is 
    \begin{equation*}
    ||y_p||_{\ell_2}^2 \triangleq 
    \sum_{k=0}^{N} \begin{bmatrix}
    x[k]^T & u[k]^T
    \end{bmatrix}  
    \begin{bmatrix}
        \sqrt{Q} & 0\\
        0 & \sqrt{R}
        \end{bmatrix}
        \begin{bmatrix}
        \sqrt{Q} & 0\\
        0 & \sqrt{R}
        \end{bmatrix} \begin{bmatrix}
        x[k] \\ u[k] 
        \end{bmatrix}
    \end{equation*}
    Thus we depict $(i)$ how to choose the nominal matrices $C_J,$ and $D_J$, and $(ii)$ the reason to evaluate the attack impact in terms of $||y_p||_{\ell_2}^2$. Similarly, we assume that faults which have a very low effect on the detection output can occur rarely (say faults which change the energy of the detection output to $\epsilon_r$). Thus an attack is considered to be detected when $||y_p||_{\ell_2}^2 > 1$.

\section{Proof of \textit{Theorem \ref{Thm_sampling}}}\label{Main_thm}
Before presenting the proof, an introduction to scenario-based constraint satisfaction is provided.
\subsection*{Scenario-based constraint satisfaction \cite{campi2018general}}
Consider the constrained non-convex optimization problem $\inf_{\theta\in \Theta}  f(\theta,\delta),$ where $\delta \in \Delta$ is the uncertainty and $\theta$ is the infinite dimensional decision parameter which lies in the set $\Theta \triangleq \bigcap_{\delta \in \Delta} \Theta_{\delta}$, where $\Theta_{\delta}$ is the constraint set which includes all the admissible parameters for the isolated uncertainty $\delta$.

\begin{definition}[Violation probability]
Let us define the violation probability as $\mathbb{V}(\theta) \triangleq \mathbb{P}_{\Delta}\{\delta \in \Delta \;|\;\theta \notin \Theta\}.$ $\hfill\triangleleft$
\end{definition}

\begin{definition}[$\epsilon$-level solution]
Let $\epsilon \in (0,1)$. Then, $\theta \in \Theta$ is an $\epsilon$ level solution if $\mathbb{V}(\theta) \leq \epsilon. \hfill\triangleleft$
\end{definition}

\begin{definition}[Confidence level : $\lambda$]
Let $\lambda \in (0,1)$. Then, the confidence level $\lambda$ represents the probability that $\theta$ is not an $\epsilon$ level solution. i.e., $\lambda \triangleq \mathbb{P}\{\mathbb{V}(\theta) > \epsilon\}.\hfill\triangleleft$ 
\end{definition}

\begin{proof}[Proof of \textit{Theorem \ref{Thm_sampling}}]
The optimization problem \eqref{fi1} can be reformulated as
\begin{equation}\label{w7}
    -\inf_{a \in \ell_{2e}} \left\{ \mathbb{E}_{\Omega} \{-||y_p(\delta)||_{\ell_2}^2 \}\left. \right|\  \begin{pmatrix}
||y_r(\delta)||_{\ell_2}^2 \leq1\\
x(\delta)[\infty]=0
\end{pmatrix} \forall \delta \in \Omega \right\}
\end{equation}
In view of \cite[Theorem 1]{campi2018general}, let us define the objective function as $f(a,\delta) \triangleq \mathbb{E}_{\Omega} \{-||y_p(\delta)||_{\ell_2}^2\},$
where $y_p(\delta)$ is also a function of the attack vector $a$. Let us define the set $\Theta$ as in \eqref{theta}. Let us define a confidence level $\lambda \in (0,1)$, a constant $N_2$ and $\epsilon(\cdot)$ such that \eqref{sample:constraint2} holds. Then applying \cite[Theorem 1]{campi2018general}, we obtain that
\begin{equation}
\mathbb{P}^{N_2}\{\mathbb{P}_{\Omega}\{\delta \in \Omega\;\vert\;a_{N_2}^* \notin \Theta\} > \epsilon(s_{N_2}^*)\}\leq \lambda,
\end{equation}
Thus, with a probability level $1-\lambda$, the solution $a_{N_2}^*$ is $\epsilon(s_{N_2}^*)$ feasible to the optimization problem \eqref{w7}. In our problem setting, $a_{N_2}^*$ is the optimal argument of the optimization problem
\begin{equation}
\begin{aligned}
-arg \inf_{a \in \ell_{2e}} \; & \left \{\frac{-1}{N_2} \sum_{i=1}^{N_2} \{||y_{p,i}||_{\ell_2}^2\; \vert \;
\begin{pmatrix}
||y_{r,i}||_{\ell_2}^2 \leq 1\\
x_{i}[\infty]=0
\end{pmatrix}\} \right\}\\
\textrm{s.t.} \quad & \begin{pmatrix}
||y_{r,i}||_{\ell_2}^2 \leq1\\
x_{i}[\infty]=0
\end{pmatrix}, \forall i \in \{1,\dots,N_2\},
\end{aligned}
\end{equation}
which can be rewritten as \eqref{uncertain:adversary}. This concludes the proof.
\end{proof}
\section{Proof of \textit{Theorem \ref{Thm2}}}\label{App2}
Before presenting the proof, an introduction to S-system and S-procedure is provided as it helps formulating the proof.
\subsection*{S-system \cite[\textit{Definition 4.3.1}]{petersen2012robust}}
Let $\mathcal{L}$ be a real Hilbert space with a well defined inner product denoted by $\langle \cdot,\cdot \rangle$. Let $\mathcal{G}_0(\cdot),\dots,\mathcal{G}_k(\cdot)$ be quadratic functionals mapping $\mathcal{L} \to \mathbb{R}$. Let $\omega$ be a discrete time signal.
\begin{definition}[S-system]\cite[\textit{Definition 4.3.1}]{petersen2012robust}\label{S-system}
The quadratic functionals $\mathcal{G}_0(\cdot),\mathcal{G}_1(\cdot),\dots,\mathcal{G}_k(\cdot)$ form an S-system if there exist a bounded linear operator $\mathbf{T}_i: \mathcal{L} \rightarrow \mathcal{L},i=1,2,\dots,$ such that 
\begin{enumerate}
    \item $\langle \mathbf{T}_i\omega_1,\omega_2\rangle \rightarrow 0$ as $i \rightarrow \infty,\;\forall\; \omega_1,\omega_2 \in \mathcal{L}$.
    \item If $\omega \in \mathcal{M},$ then $\mathbf{T}_i\omega \in \mathcal{M},\;\forall i=1,2,\dots$, where $\mathcal{M}$ is a linear subspace of $\mathcal{L}$.
    \item $\mathcal{G}_j(\mathbf{T}_i\omega) \to \mathcal{G}_j(\omega)$ as $i \rightarrow \infty,\;\forall \omega\in \mathcal{L}$ and $j=0,1,\dots,k$.$\hfill\triangleleft$
\end{enumerate}
\end{definition}
We next present a theorem which helps in proving \textit{Theorem \ref{Thm2}}.
\begin{theorem}\label{thm_new}
Let us define a stable discrete time linear time-invariant system of the form
\begin{align}\label{eq1}
    {\eta}[k+1] &= \Phi\eta[k] + \Lambda\mu[k]\\
    \sigma[k] &=\Pi\eta[k] + \Upsilon \mu[k]\qquad \eta[0]=\eta_0, \; \eta[\infty]=0.
\end{align}
Let us define the set $\mathcal{L}$ as
\begin{equation}\label{eq2}
    \mathcal{L} \triangleq \left\{  \omega = \begin{bmatrix}
    \sigma\\\mu
    \end{bmatrix} \vert \begin{bmatrix}
    \sigma, \mu \text{\;are related by\;} \eqref{eq1}\\
    \mu \in \ell_{2e}, \; \eta[0]=\eta_0, \; \eta[\infty]=0.
    \end{bmatrix} \right\}
\end{equation}
Let us also define the functionals $ \mathcal{G}_0(\omega) \triangleq \sum_{k=0}^{\infty} \omega[k]^TM_0\omega[k] + \zeta_0, \dots,   \mathcal{G}_k(\omega) \triangleq \sum_{k=0}^{\infty} \omega[k]^TM_k\omega[k] + \zeta_k$.
where $M_0,\dots,M_k$ are given matrices and $\zeta_0,\dots,\zeta_k$ are given constants. Then, the functionals $-\mathcal{G}_0(\cdot),\dots,\mathcal{G}_k(\cdot)$ form an S-system.
\end{theorem}
\begin{proof}[Proof of \textit{Theorem \ref{thm_new}}]
In view of \textit{Definition \ref{S-system}}, let us define the operator $\mathbf{T}_i$ as
\[\mathbf{T}_i\omega[k]= 
\begin{cases}
    0, & \text{if}\;\; 0\leq k \leq i\\
    \omega[k-i],  & \text{if}\;\; k > i
\end{cases}.\]
For any $\omega_1,\omega_2 \in \mathcal{L}$, $\langle \mathbf{T}_i\omega_1,\omega_2 \rangle =  |\sum_{k=0}^\infty \langle \mathbf{T}_i\omega_1[k],\omega_2[k]\rangle|^2 = $
\begin{equation}
|\sum_{k=i}^\infty \langle \omega_1[k-i],\omega_2[k]\rangle|^2 \leq_{1} \sum_{k=i}^{\infty} ||\omega_1[k-i]||^2 \sum_{k=i}^{\infty} ||\omega_2[k]||^2
\end{equation}
where the inequality $1$ holds due to the Cauchy Schwartz inequality. Since the theorem states that $\lim_{k\to \infty}\eta[k] =0$, it holds from \textit{Lemma \ref{vanish_lemma}} that $\lim_{k\to \infty}\mu[k] =0$. Following which, it immediately holds that $\lim_{k\to \infty}\sigma[k] =0$ which implies that $\lim_{k\to \infty}\omega[k] =0$. Due to the above reasoning, it holds that $\lim_{i\to \infty} \sum_{k=i}^{\infty}||\omega_2[k]||^2 \to 0 \; \forall \omega \in  \mathcal{L}$ . Thus condition $1)$ of \textit{Definition \ref{S-system}} holds. 

Let us consider a set $\mathcal{M} = \mathcal{L}|_{\eta_0=0}$. Then, if $\omega \in \mathcal{M}$, due to the time-invariant property of \eqref{eq1}, $\mathbf{T}_i\omega \in \mathcal{M}$. This proves that $\exists \;\mathcal{M} \subset \mathcal{L}$ which is invariant under the operator $\mathbf{T}_i$. Thus condition $2)$ of \textit{Definition \ref{S-system}} holds. Let us consider $\mathcal{G}_0(\mathbf{T}_i\omega) = $
\begin{align*}
\sum_{k=0}^{\infty} (\mathbf{T}_i\omega[k])^TM_0\mathbf{T}_i\omega[k] + \zeta_0 = \sum_{k=i}^{\infty} \omega[k-i]^TM_0\omega[k-i] + \zeta_0
\end{align*}
$= \mathcal{G}_0(\omega)$. Similarly, it can be observed that $\mathcal{G}_j(\mathbf{T}_i\omega) = \mathcal{G}_j(\omega)\;\forall j=\{0,\dots,k\}$. Thus condition $3)$ of \textit{Definition \ref{S-system}} holds. Since we have shown that the functionals $-\mathcal{G}_0(\omega(\cdot)),\mathcal{G}_1(\omega(\cdot)),\dots,\mathcal{G}_k(\omega(\cdot))$ satisfy conditions $1),2)$ and $3)$ of \textit{Definition \ref{S-system}}, they form an S-system. This concludes the proof.
\end{proof}
Now we are ready to present the proof of \textit{Theorem \ref{Thm2}}.

\begin{proof}
\subsubsection*{Step 1}[Problem reformulation]
Using the hypograph formulation, \eqref{uncertain:adversary} can be rewritten as 
\begin{equation}\label{eqq1}
\sup_{\upsilon, a \in \ell_{2e}} \left\{  \upsilon \; \Bigg|\; \begin{aligned}
& \frac{1}{N_2} \sum_{i=1}^{N_2} ||y_{p,i}||_{\ell_2}^2 \geq \upsilon\\
& ||y_{r,i}||_{\ell_2}^2 \leq 1,\;x_{i}[\infty]=0,\; \forall i \in \{1,\dots,N_2\}
\end{aligned}\right\}
\end{equation}
From now, in the next two steps of the proof, we focus on the optimization problem \eqref{eqq1} without the state constraints
\begin{equation}\label{eqq2}
\sup_{\upsilon, a \in \ell_{2e}} \left\{  \upsilon \; \Bigg|\; \begin{aligned}
& \frac{1}{N_2} \sum_{i=1}^{N_2} ||y_{p,i}||_{\ell_2}^2 \geq \upsilon\\
& ||y_{r,i}||_{\ell_2}^2 \leq 1,\; \forall i \in \{1,\dots,N_2\}
\end{aligned}\right\}
\end{equation}

The reason to focus on \eqref{eqq2} rather than \eqref{eqq1} is that S-procedure becomes convenient to apply when there are no equality constraints. Thus, we drop the state constraints now and introduce them back at the end of \textit{Step 3}. Equivalently \eqref{eqq2} be reformulated as
\begin{equation}\label{eq4}
\inf_{\upsilon}\left\{\upsilon \Bigg|\left\{a\in\ell_{2e}\Big|\;\begin{aligned}
&  \frac{1}{N_2} \sum_{i=1}^{N_2} ||y_{p,i}||_{\ell_2}^2 - \upsilon > 0 \\
& 1-||y_{r,i}||_{\ell_2}^2 \geq 0, \forall i 
\end{aligned}\right\}=\emptyset\right\}
\end{equation}
\subsubsection*{Step 2}
Recall that the system matrices for the system with the isolated uncertainty $\delta_i$, with attack vector as input and performance and detection outputs as outputs are
$\tilde{\Sigma}_{p,i} \triangleq ({A}_{cl,i}, {B}_{cl,i}, {C}_{p,i}, {D}_{p,i})$ and  $\tilde{\Sigma}_{r,i} \triangleq ({A}_{cl}, {B}_{cl,i}, {C}_{r,i}, {D}_{r,i})$. Let us consider a linear time-invariant system of the form \eqref{eq1} with the attack vector $a$ as input $\mu(\cdot)$ and the vector $\begin{bmatrix}
y_{p,1}&y_{r,1},\dots,y_{pN_2}&y_{r,N_2}\end{bmatrix}^T$ as output $\sigma(\cdot)$. This system will be stable due to \textit{Assumption \ref{assume_stable}} and the system matrices of \eqref{eq1} would read
\begin{align}
\left[
\begin{array}{c|c}
\Phi & \Lambda\\
\hline
\\[\dimexpr-\normalbaselineskip+2pt] \Pi & \Upsilon
\end{array}
\right] = 
\left[ 
\begin{array}{c|c}
\begin{matrix}
A_{cl,1}& 0 & 0 \\
\; & \ddots & \;\\
0 & 0 & A_{cl,N_2}
\end{matrix}     & \begin{matrix}
B_{cl,1} \\ \vdots \\ B_{cl,N_2}
\end{matrix} \\
\hline
\\[\dimexpr-\normalbaselineskip+2pt] \begin{matrix}
C_{p,1} & 0 & 0\\
C_{r,1} & 0 & 0\\
\; & \ddots & \; \\
0 & 0 & C_{p,N_2}\\
0 & 0 & C_{r,N_2}
\end{matrix} & \begin{matrix}
D_{p,1}\\
D_{r,1}\\
\vdots\\
D_{p,N_2}\\
D_{r,N_2}
\end{matrix}
\end{array}
\right],
\end{align}
For this system, let us define the set $\mathcal{L}$ as in \eqref{eq2} where $\omega = \begin{bmatrix} a^T& \sigma \end{bmatrix}^T \in \mathbb{R}^{n_a+N_2(n_p+n_r)}$. In view of \eqref{eq4}, let us also define the functionals $\mathcal{G}_0(\omega) \triangleq $
\begin{equation}\label{G_0}
\frac{1}{N_2} \sum_{i=1}^{N_2} ||y_{p,i}||_{\ell_2}^2 - \upsilon= \sum_{k=0}^{\infty} \omega[k]^TM_0\omega[k] + \zeta_0
\end{equation}
where $M_0 \in \mathbb{R}^{(n_a+N_2(n_p+n_r)) \times (n_a+N_2(n_p+n_r))}, \zeta_0=-\upsilon$ and $M_0(i,j)= 1, \text{if}\; i=j, n_a+i$  is odd and $0$ elsewhere. Similarly let us define 
\begin{equation}\label{G_1}
    \mathcal{G}_k(\omega) \triangleq  -||y_{r,k}||_{\ell_2}^2+1 \;=\;\sum_{k=0}^{\infty} \omega[k]^TM_k\omega[k] + \zeta_k,
\end{equation}
$\forall k =\{1,2,\dots,N_2\}$, where $\zeta_k = -1$ and $M_k(i,j)= -1, \text{if}\; i=j, i=n_a+2k-1$ and $0$ elsewhere. Here, $M_k$ has same dimension as $M_0\;\forall k =\{1,2,\dots,N_2\}$. Therefore, we have shown that the constraints of \eqref{eq4} can be rewritten as functionals of the set $\mathcal{L}$. We can now see that the functionals $-\mathcal{G}_0(\cdot),\mathcal{G}_1(\cdot),\dots,\mathcal{G}_k(\cdot)$ along with \textit{Lemma \ref{vanish_lemma}} satisfy the conditions under which \textit{Theorem \ref{thm_new}} holds. Thus, by applying \textit{Theorem \ref{thm_new}}, if follow that the functionals $-\mathcal{G}_0(\cdot),\mathcal{G}_1(\cdot),\dots,\mathcal{G}_k(\cdot)$ form an S-system. Let this be argument 1.

In the case the adversary chooses not to attack the system, i.e., $a=0 \in \ell_{2e}$, it follows that $||y_{r,i}||_{\ell_2}^2 \approx 0 \;\forall \delta_i$. The residual energy $||y_{r,i}||_{\ell_2}^2$ is not strictly zero since there might be residual outputs due to difference in initial condition between the system and the detector. The threshold ($\epsilon_r = 1$) is chosen in such a way that $||y_{r,i}||_{\ell_2}^2 \ll 1\; \forall \delta_i$ when $a=0$. Thus, it holds that $\exists \; \omega_0 = [a,\;\omega] = [0,\; 0_{+}] $ s.t. $ -||y_{r,k}||_{\ell_2}^2+1 = \mathcal{G}_k(\omega_0) >0\; \forall k=\{1,\dots,N_2\}$. Here $0_{+}$ represents a real number close to zero. Let this be argument 2. And from \eqref{eq4}, we know that the system $\mathcal{G}_0(\omega) > 0 , \mathcal{G}_i(\omega) \geq 0 \;i=\{1,\dots,k\}$ is not solvable. Let this be argument 3. Using the above arguments (1-3) and \cite[Theorem 4.3.1]{petersen2012robust}, we can conclude that when we consider the functionals defined in \eqref{G_0} and \eqref{G_1}, $\exists \; \gamma_1 \geq 0,\gamma_1 \geq 0,\dots,\gamma_{N_2}\geq 0$ such that the following inequality holds
\begin{equation}\label{eq5}
 \mathcal{G}_0(\omega) + \sum_{i=1}^{N_2}\gamma_i\mathcal{G}_i(\omega) \leq 0,\; \forall \; \omega \in \mathcal{L}.
\end{equation}
To conclude, in this step we have shown that the constraint of \eqref{eq4} holds only if \eqref{eq5} is true. And the converse is generally true.
\subsubsection*{Step 3}
We have shown that the constraint of \eqref{eq4} holds iff \eqref{eq5} is true. Then, we reformulate \eqref{eq4} as 
\begin{equation}
\inf_{\upsilon,\gamma_1\geq 0,\dots,\gamma_{N_2}\geq 0} \left\{ \upsilon \;\Big|\; \mathcal{G}_0(\omega) + \sum_{i=1}^{N_2}\gamma_i\mathcal{G}_i(\omega) \leq 0,\; \forall \; \omega(\cdot) \in \mathcal{L}\right\}. 
\end{equation}
Substituting the definition of $\mathcal{G}_0(\omega)$, we obtain
\begin{equation}\label{eq6}
\inf_{\gamma\geq 0}\left\{ \inf_{\upsilon} \left\{\upsilon\Big|\sum_{i=1}^{N_2} \left\{ \frac{1}{N_2}||y_{p,i}||_{\ell_2}^2\; + \gamma_i\mathcal{G}_i(\omega)\right\} \leq \upsilon,\forall\omega \right\}\right\}
\end{equation}
where $\gamma=[\gamma_1,\dots,\gamma_{N_2}]^T$. The inner optimization problem of \eqref{eq6} resembles an epigraph formulation which can be rewritten as
\begin{equation}
\inf_{\gamma\geq 0}\left\{
\begin{aligned}
\sup_{\omega} \; & \left\{\sum_{i=1}^{N_2} \left\{ \frac{1}{N_2}||y_{p,i}||_{\ell_2}^2\;
 + \gamma_i\mathcal{G}_i(\omega)\right\}\right\}
\end{aligned}\right\}.
\end{equation}
Substituting the definition of $\mathcal{G}_i(\omega),\;\forall i=\{1,\dots,N_2\}$, we obtain
\begin{equation}\label{eq7}
\inf_{\gamma\geq 0}\underbrace{
\sup_{\omega} \left\{\sum_{i=1}^{N_2} \left\{ \frac{1}{N_2}||y_{p,i}||_{\ell_2}^2-\gamma_i||y_{r,i}||_{\ell_2}^2\right\}+\sum_{i=1}^{N_2} \gamma_i \right\} 
}_{\kappa}
\end{equation}
Observe that $\kappa$ is a maximization problem with a quadratic term in its objective. Thus, it holds that
\[
\kappa = 
\begin{cases}
    \mathbf{1}^T \gamma, & \text{if}\;\;\sum_{i=1}^{N_2} \left\{ \frac{1}{N_2}||y_{p,i}||_{\ell_2}^2\;
 - \gamma_i||y_{r,i}||_{\ell_2}^2\right\} \leq 0\\
    +\infty,  & \text{otherwise}
\end{cases}.
\]
Using the above result in \eqref{eq7}, we obtain
\begin{equation*}
\inf_{\gamma \geq 0}\; \left\{ \mathbf{1}^T \gamma \Bigg| \sum_{i=1}^{N_2} \frac{1}{N_2}||y_{p,i}||_{\ell_2}^2 - \gamma_i ||y_{r,i}||_{\ell_2}^2 \leq 0, \forall a \in \ell_{2e}\right\}
\end{equation*}
where $\omega$ is replaced by $a$ since it is only the control variable. Finally, we add the state constraints that were removed while formulating the optimization problem \eqref{eqq2}
\begin{equation}\label{eqq3}
\inf_{\gamma \geq 0} \left\{ \mathbf{1}^T \gamma \;\Bigg| \begin{aligned}
&\sum_{i=1}^{N_2} \frac{1}{N_2}||y_{p,i}||_{\ell_2}^2 - \gamma_i ||y_{r,i}||_{\ell_2}^2 \leq 0, \forall a \in \ell_{2e}\\
&x_{i}[\infty]=0 \;\forall i \in \{1,\dots,N_2\} \end{aligned}\right\}
\end{equation}
Thus in this step, we have shown using S-procedure that the optimization problem \eqref{eqq1} and \eqref{eqq3} are equivalent. 
\subsubsection*{Step 4}
Define $\bar{x}[\infty] =$ 
$ \begin{bmatrix} x_{1}[\infty]^T & \dots & x_{N_2}[\infty]^T \end{bmatrix}^T, \bar{x}[0] = \begin{bmatrix} x_{1}[0]^T & \dots & x_{N_2}[0]^T \end{bmatrix}^T$, \\
$\bar{y}_p = \begin{bmatrix} y_{p1}^T & \dots & y_{pN_2}^T
\end{bmatrix}^T$, and $\bar{y}_r = \begin{bmatrix} y_{r1}^T & \dots & y_{rN_2}^T \end{bmatrix}^T$. Using these definitions, the constraint of \eqref{eqq3} can be rewritten as 
\begin{align}\label{con1}
-\frac{1}{N_2}||\bar{y}_p||_{\ell_2}^2 + ||\sqrt{\Gamma(\gamma)}\bar{y}_r||_{\ell_2}^2 \geq 0,  \forall a \in \ell_{2e},\;\; \bar{x}[\infty]=0
\end{align}
where $\Gamma(\gamma)$ is defined in the theorem statement. Additionally due to \textit{Assumption \ref{equil_begin}}, we have $\bar{x}[0] =0$. Next let us define $y_1 = \sqrt{\Gamma(\gamma)}\bar{y}_r, y_2=\frac{1}{\sqrt{N_2}}\bar{y}_p$ and the supply rate $s(\cdot) \triangleq  ||{y_1}||_{2}^2 -||{y_2}||_{2}^2$. Then, we have shown from \eqref{con1} that \cite[\textit{Proposition 2, 2)}]{truncated} holds. Equivalently, using \cite[\textit{Proposition 2, 3)}]{truncated}, we replace \eqref{con1} by the constraint of \eqref{problem:coupled:final} where $\bar{\Sigma}_p = (\bar{A},\bar{B},\bar{C}_p,\bar{D}_p)$ and $\bar{\Sigma}_r = (\bar{A},\bar{B},\bar{C}_r,\bar{D}_r)$ represent the system with the attack as input and $\bar{y}_p$ and $\bar{y}_r$ as system outputs respectively. Constructing these system matrices, as we did in \textit{Step 2} of this proof concludes the proof.
\end{proof}


\section{Proof of \textit{Lemma \ref{Bound_RA}}}\label{bound_2}

\begin{proof}
To recall, the optimization problem \eqref{problem:coupled:final} was formulated using \cite[\textit{Proposition 2, 3)}]{truncated} where $y_1=\sqrt{\Gamma(\gamma)} \bar{y}_r$ and $y_2= \frac{1}{\sqrt{N_2}}\bar{y}_p$. Here $\bar{y}_p$ and $\bar{y}_r$ represents the outputs of $\bar{\Sigma}_p$ and $\bar{\Sigma}_r$ respectively. Due to the equivalency between $3)$ and $4)$ of \cite[\textit{Proposition 2, 3)}]{truncated}, the FDI \cite[\textit{Proposition 2, 4)}]{truncated} should hold $\;\forall \;|z|=1$. Since we know that $y_1 = \sqrt{\Gamma(\gamma)} \bar{y}_r$ and $y_2 = \bar{y}_p$, we can deduce that $G_1(z)$ corresponds to $\sqrt{\Gamma(\gamma)}\bar{G}_{r}(z)$ and $G_2(z)$ to $\frac{1}{N_2}\bar{G}_{r}(z)$ in \cite[\textit{Proposition 2, 4)}]{truncated}, where $\bar{G}_{r}(z)= \bar{C}_{r}(z_1I-\bar{A})^{-1}\bar{B}+\bar{D}_{r}$ and $\bar{G}_{p}(z) \triangleq \bar{C}_{p}(z_1I-\bar{A})^{-1}\bar{B}+\bar{D}_{p}$. Thus, \eqref{problem:coupled:final} can be rewritten as
\begin{equation}\label{lem_s2}
\inf\Big\{ \textbf{1}^T\gamma\Big|\bar{G}_{r}(\bar{z})^T\Gamma(\gamma)\bar{G}_{r}({z})-\bar{G}_{p}^T(\bar{z})\bar{G}_{p}({z})\succeq 0,\forall|z|=1\Big\}
\end{equation}
Let us define the following sets such that $\mathbb{C}^{n_a} = \mathcal{Z}_{pr} \cup \mathcal{Z} \cup \mathcal{Z}_{r} \cup \mathcal{Z}_{p}$.
\begin{align}
    \mathcal{Z}_{pr} & \triangleq \{x \in \mathbb{C}^{n_a} \;|\; \bar{G}_{r}({z})x = 0 , \bar{G}_{p}({z})x = 0\},\\
    \mathcal{Z} & \triangleq \{x \in \mathbb{C}^{n_a} \;|\; \bar{G}_{r}({z})x \neq 0 , \bar{G}_{p}({z})x \neq 0\},\\
    \mathcal{Z}_r & \triangleq \{x \in \mathbb{C}^{n_a} \;|\; \bar{G}_{r}({z})x = 0 , \bar{G}_{p}({z})x \neq 0\},\\
    \mathcal{Z}_p & \triangleq \{x \in \mathbb{C}^{n_a} \;|\; \bar{G}_{r}({z})x \neq 0 , \bar{G}_{p}({z})x = 0\}.
\end{align}
\textit{Sufficiency:} For any given $z$ such that $|z|=1$ , if $ x \in \mathcal{Z}_p$ or $ x \in \mathcal{Z}_{pr}$, choosing $\Gamma(\gamma) =0$ satisfies the constraint of \eqref{lem_s2}. Similarly, if $ x \in \mathcal{Z}$, let us pick $\Gamma(\gamma) = cI_{n_{r}}$ where c is a constant. Then, the value of \eqref{lem_s2} is bounded if there exists a bounded $c$ that makes $ \Xi \triangleq \sup_{|z|=1, x \in \mathcal{Z}} \frac{x^H\big\{\bar{G}_{r}^T(\bar{z})\bar{G}_{r}({z})\big\}x}{x^H\big\{\bar{G}_{p}^T(\bar{z})\Gamma(\gamma)\bar{G}_{p,i}({z})\big\}x}$ bounded. $\Xi$ is bounded  since the denominator cannot become zero (since $x \in \mathcal{Z}$ and $\Gamma(\gamma)$ is full rank), and we have assumed that the $\bar{G}_{r}({z})$ and $\bar{G}_{p}({z})$ are stable (\textit{Assumption \ref{assume_stable}}). Next we prove sufficiency when $x \in \mathcal{Z}_r$.

When condition $1)$ of the lemma statement holds, by definition of a zero $ \forall |z|=1, \nexists s \neq 0 \in \mathbb{C}^{n_a}$ such that $\bar{G}_{r}({z})s = 0$. Thus it follows that $\mathcal{Z}_r = \mathcal{Z}_{pr} = \emptyset$. When condition $2)$ of the lemma statement holds, by definition of a zero $ \forall |z|=1, \nexists s \neq 0$ such that $\bar{G}_{r}({z})s = 0$ and $\bar{G}_{p}({z})s \neq 0$. Thus it follows that $\mathcal{Z}_r = \emptyset$. \\
\textit{Necessity:} Assume that there exists a bounded $\Gamma(\gamma)$ that solves \eqref{lem_s2}. We also assume that there exists a complex number $z_1$ on the unit circle which is a zero of the system $\bar{\Sigma}_{r}$ (including multiplicity and input direction) but are not zeros of $\bar{\Sigma}_{p}$. By definition of a zero, it holds that $\exists s \neq 0$ such that $\bar{G}_{r}(z_1)s=0$ and $\bar{G}_{p}(z_1)s \neq 0$. Under these assumptions, when $z=z_1$ and $x = s$, the constraint of \eqref{lem_s2} can be rewritten as $-s^H\bar{G}_{p}^T(\bar{z}_1)\bar{G}_{p,i}({z_1})s \geq 0$ which cannot hold since $\bar{G}_{p}({z_1})s \neq 0$. That is, the feasibility set of \eqref{lem_s2} is empty which contradicts our assumption. This concludes the proof.
\end{proof}

%% file: ms.bbl
\begin{thebibliography}{10}

\bibitem{dibaji2019systems}
S.~M. Dibaji, M.~Pirani, D.~B. Flamholz, A.~M. Annaswamy, K.~H. Johansson, and
  A.~Chakrabortty, ``A systems and control perspective of {CPS} security,''
  {\em Annual Reviews in Control}, vol.~47, pp.~394--411, 2019.

\bibitem{case2016analysis}
D.~U. Case, ``Analysis of the cyber attack on the ukrainian power grid,'' {\em
  Electricity Information Sharing and Analysis Center}, vol.~388, 2016.

\bibitem{hemsley2018history}
K.~E. Hemsley, E.~Fisher, {\em et~al.}, ``History of industrial control system
  cyber incidents,'' tech. rep., Idaho National Lab, Idaho Falls, U.S., 2018.

\bibitem{risk}
{\em Guide for conducting risk assessment}.
\newblock N.I.S.T., 2012.

\bibitem{cardenas2011attacks}
A.~A. C{\'a}rdenas, S.~Amin, Z.-S. Lin, Y.-L. Huang, C.-Y. Huang, and
  S.~Sastry, ``Attacks against process control systems: risk assessment,
  detection, and response,'' in {\em Proceedings of the 6th ACM symposium on
  information, computer and commun. security}, pp.~355--366, 2011.

\bibitem{CVAR1}
Y.~W. Law, T.~Alpcan, and M.~Palaniswami, ``Security games for risk
  minimization in automatic generation control,'' {\em IEEE Transactions on
  Power Systems}, vol.~30, no.~1, pp.~223--232, 2014.

\bibitem{teixeira2015secure}
A.~Teixeira, K.~C. Sou, H.~Sandberg, and K.~H. Johansson, ``Secure control
  systems: A quantitative risk management approach,'' {\em IEEE Control Systems
  Magazine}, vol.~35, no.~1, pp.~24--45, 2015.

\bibitem{pan2018cyber}
K.~Pan, A.~Teixeira, M.~Cvetkovic, and P.~Palensky, ``Cyber risk analysis of
  combined data attacks against power system state estimation,'' {\em IEEE
  Transactions on Smart Grid}, vol.~10, no.~3, pp.~3044--3056, 2018.

\bibitem{park2019stealthy}
G.~Park, C.~Lee, H.~Shim, Y.~Eun, and K.~H. Johansson, ``Stealthy adversaries
  against uncertain cyber-physical systems: Threat of robust zero-dynamics
  attack,'' {\em IEEE Trans. on Automat. Contr.}, vol.~64, no.~12,
  pp.~4907--4919, 2019.

\bibitem{harshbarger2020little}
S.~Harshbarger, M.~Hosseinzadehtaher, B.~Natarajan, E.~Vasserman, M.~Shadmand,
  and G.~Amariucai, ``({A} {L}ittle) ignorance is bliss: The effect of
  imperfect model information on stealthy attacks in power grids,'' in {\em
  2020 IEEE Kansas Power and Energy Conf.}, pp.~1--6, IEEE, 2020.

\bibitem{murguia2020security}
C.~Murguia, I.~Shames, J.~Ruths, and D.~Ne{\v{s}}i{\'c}, ``Security metrics and
  synthesis of secure control systems,'' {\em Automatica}, vol.~115, p.~108757,
  2020.

\bibitem{Andre1}
A.~Teixeira, I.~Shames, H.~Sandberg, and K.~H. Johansson, ``A secure control
  framework for resource-limited adversaries,'' {\em Automatica}, vol.~51,
  pp.~135--148, 2015.

\bibitem{milovsevic2020actuator}
J.~Milo{\v{s}}evi{\'c}, A.~Teixeira, K.~H. Johansson, and H.~Sandberg,
  ``Actuator security indices based on perfect undetectability: Computation,
  robustness, and sensor placement,'' {\em IEEE Trans. on Automat. Control},
  2020.

\bibitem{teixeira2015strategic}
A.~Teixeira, H.~Sandberg, and K.~H. Johansson, ``Strategic stealthy attacks:
  the output-to-output $\ell_2$-gain,'' in {\em 2015 54th IEEE Conference on
  Decision and Control (CDC)}, pp.~2582--2587, IEEE, 2015.

\bibitem{IFAC}
S.~C. Anand and A.~M. Teixeira, ``Joint controller and detector design against
  data injection attacks on actuators,'' {\em IFAC-PapersOnLine}, vol.~53,
  no.~2, pp.~7439--7445, 2020.

\bibitem{ferrari2021safety}
R.~M. Ferrari and A.~M. Teixeira, ``Safety, security, and privacy for
  cyber-physical systems,'' 2021.

\bibitem{lye2005game}
K.~W. Lye and J.~M. Wing, ``Game strategies in network security,'' {\em Intl.
  Journal of Information Security}, vol.~4, no.~1-2, pp.~71--86, 2005.

\bibitem{ben1998robust}
A.~Ben-Tal and A.~Nemirovski, ``Robust convex optimization,'' {\em Mathematics
  of operations research}, vol.~23, no.~4, pp.~769--805, 1998.

\bibitem{tempo2012randomized}
R.~Tempo, G.~Calafiore, and F.~Dabbene, {\em Randomized algorithms for analysis
  and control of uncertain systems: with applications}.
\newblock Springer Science \& Business Media, 2012.

\bibitem{campi2018general}
M.~C. Campi, S.~Garatti, and F.~A. Ramponi, ``A general scenario theory for
  nonconvex optimization and decision making,'' {\em IEEE Transactions on
  Automatic Control}, vol.~63, no.~12, pp.~4067--4078, 2018.

\bibitem{petersen2012robust}
I.~R. Petersen, V.~A. Ugrinovskii, and A.~V. Savkin, {\em Robust Control Design
  Using $H_{\infty}$ Methods}.
\newblock Springer Science \& Business Media, 2012.

\bibitem{truncated}
A.~M. Teixeira, ``Optimal stealthy attacks on actuators for strictly proper
  systems,'' in {\em 2019 IEEE 58th Conf. on Decision and Contr. (CDC)},
  pp.~4385--4390, IEEE, 2019.

\end{thebibliography}
